\documentclass{article}
\usepackage[latin1]{inputenc}
\usepackage{amssymb}
\usepackage{amsmath,amstext}

\newcommand{\pa}{{\partial }}
\usepackage{amsthm}
\usepackage{amsfonts}

\newcommand{\cD}{{\cal D}}

\newcommand{\cN}{{\cal N}}

\newcommand{\cR}{{\cal R}}

\newcommand{\cU}{{\cal U}}
\newcommand{\cV}{{\cal V}}

\newcommand{\R}{\ensuremath{\mathbb{R}}}

\newtheorem{Theorem}{Theorem}

\def\eref#1{(\ref{#1}%
)}
\def\RSref#1{\ref{#1}%
}
\def\RSlabel#1{\label{#1}%
}
\def\RScite#1{\cite{#1}%
}
\def\biglf{\par\bigskip\noindent} % to be changed for journal style
\def\fa{\hbox{ for all }}
\def\bql#1{\begin{equation}\label{#1}}
\def\eq{\end{equation}}
\begin{document}
\begin{center}
  {\bf
    A Nonlinear Discretization Theory \\
    for Meshfree Collocation Methods\\
applied to  Quasilinear Elliptic Equations\\}~\\~\\

Klaus B\"ohmer\footnote{Fachbereich Mathematik und Informatik,
Universit\"at Marburg,
Arbeitsgruppe Numerik,
Hans Meerwein Stra\ss{}e, Lahnberge,
D-35032 Marburg,
Germany},
  Robert Schaback \footnote{Institut f\"ur Numerische und Angewandte Mathematik,
Universit\"at G\"ottingen,
Lotzestra\ss{}e 16-18,
D-37073 G\"ottingen,
Germany}
\end{center}
%\magenta{Stand: \today}
\begin{abstract}\par\noindent  
  We generalize 
  our earlier  results  concerning meshfree collocation  methods
  for semilinear elliptic second order  problems to
  the quasilinear case. 
  The stability question, however, is treated differently, namely by
  extending a paper on uniformly stable discretizations
  of well-posd linear problems to the nonlinear case.
  These two ingredients allow a proof that all well-posed quasilinear
  elliptic second-order problems can be discretized in a
  uniformly stable way by using sufficient oversampling, and then the
  error of the numerical solution behaves like the 
  error obtainable by direct approximation of the true solution
  by functions from the chosen trial space,
  up to a factor induced by being forced to use
    a Hölder-type theory for the nonlinear PDE.
  We apply our general technique
    to prove convergence of meshfree methods
    for quasilinear elliptic equations with Dirichlet and non-Dirichlet 
  boundary conditions. This is achieved   for bifurcation and center manifolds of elliptic partial differential equations  and their numerical methods as well.
 \end{abstract}

%****************************************************************
\section{Overview}\RSlabel{SecOver}

Similarly to our previous papers, we  examine the convergence
behavior of suitable collocation-based meshfree methods.
In \RScite{BoScha12,boehmer-schaback:2017-1} we had studied
exactly one of the simplest and then one of the
most complicated nonlinear elliptic boundary value problems of order two,
the fully nonlinear  Monge--Amp\`{e}re equation.
The positive numerical experience with these cases  calls for a
generalization to a whole class of quasi-linear  problems.

For application aspects of these problems,
there are many numerical results with meshfree methods,
however, without discussing the convergence 
% \textbf{Robert, can you please add a few things there}
\RScite{%schaback-wendland:1999-1, wendland:1999-2,
  fasshauer:2002-1,zhang-et-al:2008-1,
  tsai:2012-1,davydov-saaed:2013-1,%farrell-wendland:2013-1,
  li-liu:2017-1,
  jankowska-et-al:2018-1}.
\biglf
The book \RScite{Boehmer05} contains
a lot of new, but also summarizes many known results for the non-meshfree case.
For example, for 
Finite Element Methods Zeidler \RScite{Zeidler90a} and
Skrypnik \RScite{Skrypnik86}, for difference methods Schumann and Zeidler
\RScite{SchuZe79}.
Most interesting, however, are the contributions to
Discontinuous Galerkin Methods.
They start with Rivi\'ere and Wihler \RScite{RiWh00} and continue
e.g. with S\"uli and his colleagues
Houston, Robson, and  Wihler \RScite{HoSu05,HoSuWi08,HoRoSu05}.
\biglf
This paper starts with the essential
prerequisites of the nonlinear discretization theory
from \RScite{BoScha12}, namely
\begin{enumerate}
   \item well-posedness of the PDE problem,
   \item approximation in trial spaces,
   \item testing by equations and solving by optimization, and finally 
   \item Stability.
\end{enumerate} 
However, it takes a different path in
Section \RSref{SecStab} concerning stability.
There, the general nonlinear discretization theory of \RScite{BoScha12}
  is combined with a result of \RScite{schaback:2015-4} for linear problems.
The latter
allows to construct uniformly
stable monotone, refinable, and dense discretizations of well-posed linear
problems
that are uniformly stable, the numerical stability being only dependent on
the stability of the original linear PDE problem. The region of the local validity
of this discretization does not vary with the refinement of the discretization.
This strengthens a result in \RScite{BoScha12},
and is applicable to all
PDE problems that satisfy the well-posedness assumptions of \RScite{BoScha12}.
In section \RSref{SecExa}, we prove that these assumptions are satisfied for
general elliptic quasilinear equations of order two. 
Extensions will allow other nonlinear well-posed PDE problems,
provided that the specific well-posedness assumptions are satisfied. 
%RS An additional section
Finally, Section \RSref{SecOBC} seems to be the first dealing 
  with the case of non-Dirichlet boundary conditions for meshless methods
  solving quasilinear elliptic problems.
%****************************************************************
%\section{A Collocation  Discretization Theory}\RSlabel{SecDT}
\section{Well-posed nonlinear problems}%
\RSlabel{ss:WPP}
We consider boundary value problems
on bounded Lipschitz domains $\Omega$  in $\R^d$ and formulate them strongly
as 
$$
\begin{array}{rclcl}
  Gu&=& f_1 & \hbox{ on } &\Omega \\
  Bu&=& f_2 & \hbox{ on } &\partial \Omega 
\end{array}
$$
with differential and boundary operators $G$ and $B$, respectively.
These are mappings defined as
$$
\begin{array}{rclclcl}
G&:&\cD(G) &\subseteq& \cU&\to& \cV_1\\
B&:&\cD(B) &\subseteq& \cU&\to& \cV_2\\
\end{array}
$$
on a Banach space $\cU$ of functions on $\overline{\Omega}$,
and map that space into Banach spaces
$\cV_1$ and $\cV_2$ on the domain and the boundary, respectively.
We combine the two maps and simplify the problem by
\bql{eqFGB}
\begin{array}{rcl}
  F&=& (G,B)\\
  F\;:\;\cD(F) &\subseteq& \cU\to \cV:=\cV_1\times\cV_2,\\
  Fu&=& f  := (f_1,f_2)
\end{array}
\eq
for notational convenience.
\biglf
The space $\cU$ should contain a locally unique true
solution denoted by $u^*$ that we want to approximate numerically.
\biglf
Around this local solution, we require some form of well-posedness
of the problem. One way is to let the linearization $F'$ of $F$ near $u^*$ 
be boundedly invertible as a map
$F'\;:\;\cU\to\cV$.
The other, less popular one is to ask for an inequality of the form
\bql{eq1a}
c_F^{-1} \|u-v\|_{\cU}
\leq 
\|Fu-Fv\|_{\cV}\leq C_F
\|u-v\|_{\cU}
\eq
for all $u,\,v$ in a neighborhood of $u^*$.
Fortunately, \eref{eq1a} follows from the usual well-posedness condition:
\begin{Theorem}\RSlabel{Theold3}\RScite{BoScha12}
  Let $F$ be Fr\'echet-differentiable in each point
  of a neighborhood $\cN(u^*)$ of $u^*$ and let the Fr\'echet derivatives
  $F'(u)$ at $u$
  be bounded and Lipschitz continuous, i.e.
  $$
  \|F' (v) - F'( u )\|_\cV \leq C \|u-v\|_\cU \fa u,v\in \cN(u^*).
  $$
  Finally, let $F'(u^*)$  have a bounded inverse.
  Then \eref{eq1a} holds in a neighborhood of $u^*$,
  and all Fr\'echet derivatives are uniformly bounded
  and have uniformly bounded inverses there.
\end{Theorem} 
%%%%%%%%%%%%%%%%%%%%%%%%%%%%%%%%%%%%%%%%%%%%
%%%%%%%%%%%%%%%%%%%%%%%%%%%%%%%%%%%%%%
\section{Approximation in Trial Spaces}\RSlabel{ss:AiTS}
Our meshless numerical approximations are taken from a scale
$\{U_r\}_{r>0}$ 
of linear finite-dimensional nested {\em trial} spaces
$U_r \subset \cU$ with $U_{r'} \subset U_r$  for $r ' < r$.
Our substitute for {\em consistency} is the
assumption that the true solution $u^*$ can
be approximated well by elements $u_r$ of the trial spaces
$U_r$ in the sense
\bql{eqepsr}
\inf_{u_r\in U_r} \| u^* - u_r\|_\cU \leq  \epsilon(r, u^* )
\eq
for all $r > 0$, with small $\epsilon(r, u^* )$ tending to zero for
$r\to 0$. Since we have the
norm of $\cU$ on the left-hand side,
and since we want a good convergence rate of
the approximations, we shall usually have to assume that the
true solution $u^*$
and the trial spaces $U_r$ lie in a {\em regularity subspace} $\cU_R$ of $\cU$
that determines the
convergence rate.
\biglf
Since we did not choose $\epsilon(r, u^* )$
minimally in \eref{eqepsr},
we can assume that
there are elements $u_r^* \in U_r$ with
\bql{eqapperr}
\inf_{u_r\in U_r} \| u^* - u_r\|_\cU \leq  \|u^* - u_r^*\|_\cU\leq \epsilon(r, u^* )
\eq
that realize the optimal convergence rate of approximation. Our goal is that
the numerical solution $\tilde u_r\in U_r$ of the PDE problem,
if we use a proper numerical algorithm for PDE solving
  based on the trial space $U_r$,
converges to the true solution $u^*$
as fast as the approximation $u_r^*$ in the trial space converges.
However, as convergence is usually an interplay of consistency and
  stability,
there may be factors arising from instabilities that deteriorate that
convergence rate.
%%%%%%%%%%%%%%%%%%%%%%%%%%%%%%%%%%%%%%%%%%%%%%%%
%%%%%%%%%%%%%%%%%%%%%%%%%%%%%%%%%%%%%%%%%%%%%%%
%%%%%%%%%%%%%%%%%%%%%%%%%%%%%%%%%%%%%%%%%%%%%%%%
%%%%%%%%%%%%%%%%%%%%%%%%%%%%%%%%%%%%%%%%%%%%%%%
\section{Testing and Solving}\RSlabel{SecTestSol}
The trial space discretizes the {\em domain} of $F$, but {\em testing}
discretizes the {\em range}.
This is done by a scale
$\{T_s\}_{s>0}$ of linear {\em test maps} that takes functions $f=(f_1,f_2)$
in $\cV=\cV_1\times\cV_2$ into {\em data values}
$T_s(f)\in V_s$ in some finite-dimensional space $V_s$, e.g.
$$
T_s(f)=(f_1(x_1),\ldots,f_1(x_M),f_2(y_1),\ldots,f_2(y_N))\in V_s:=\R^{M+N}
$$
with $M, \;N$, and the {\em collocation points} $x_i\in \Omega,\;y_j\in\partial\Omega$
implicitly depending on $s$.
\biglf
The {\em discretized} problem replacing $Fu=f$ now consists in solving
\bql{eqGenDisPro}
T_s(Fu_r)\approx T_s(f) 
\eq
for some trial function $u_r\in U_r$. From Section \RSref{ss:AiTS}
we know that there are good approximations $u_r^*\in U_r$ to $u^*$, and
therefore we are satisfied with finding a numerical solution
$\tilde u_r\in U_r$ with
$$
\|T_s(F\tilde u_r)-T_s(f)\|_{V_s}\leq
2\|T_s(Fu_r^*)-T_s(f)\|_{V_s}.
$$
This can, for instance, be accomplished by an approximate solution of the
finite-dimensional nonlinear optimization
problem 
$$
\min_{u_r\in U_r}\|T_s(Fu_r)-T_s(f)\|_{V_s}
$$
under a suitable parametrization of the trial space. 
Note that we do not solve the linearized problem, in contrast to many
standard algorithms for nonlinear problems.
%%%%%%%%%%%%%%%%%%%%%%%%%%%%%%%%%%%%%%%%%%%%%%%%
%%%%%%%%%%%%%%%%%%%%%%%%%%%%%%%%%%%%%%%%%%%%%%%
%%%%%%%%%%%%%%%%%%%%%%%%%%%%%%%%%%%%%%%%%%%%%%%%
%%%%%%%%%%%%%%%%%%%%%%%%%%%%%%%%%%%%%%%%%%%%%%%
%%%%%%%%%%%%%%%%%%%%%%%%%%%%%%%%%%%%%%%%%%%%%%%
\section{Stability}\RSlabel{SecStab}
Clearly, solving \eref{eqGenDisPro}  in the linear case
will run into problems if $\dim U_r>\dim V_s$, and in general it will
stabilize the problem when we {\em oversample}, i.e. take
$\dim U_r$ smaller than $\dim V_s$. We can mimic the logic
of well-posedness of Section \RSref{ss:WPP} by asking for an inequality
\bql{eq1b}
\|u_r-v_r\|_{\cU}
\leq 
C_{SF}(r,s)\|T_sFu_r-T_sFv_r\|_{V_s}
\eq
for all $u_r,\,v_r$ in $U_r$, but true stability would mean that
the constant $C_{SF}(r,s)$ has a fixed upper bound. This can be achieved
\RScite{schaback:2015-4} for well-posed problems by
letting the test strategy {\em oversample}, i.e. by letting the test maps
$T_s$ depend on the $U_r$ such that  $C_{SF}(r,s(r))$ is uniformly bounded.
\biglf
We want to play this back to the linearization, using
\RScite{BoScha12} and \RScite{schaback:2015-4}.
The latter paper proves that for all well-posed linear problems
one can find well-designed
{\em monotonic refinable dense} (MRD) discretizations that
are uniformly stable. In the notation of this paper,
the basic requirement for MRD discretizations can be shortly rephrased as  
\bql{eqvVDens}
\|v\|_\cV:=\displaystyle{\sup_{s>0} \|T_sv\|_{V_s}   }\fa v\in \cV, 
\eq
defining a norm on $\cV$. This holds if for $s\to 0$
the discrete norms
approximate the continuous norm from below
({\em monotonicity}) by getting finer and
finer ({\em refinability}) and {\em dense} in the limit. Note that this property
is independent of PDEs and trial spaces. It just expresses how discrete norms
approximate continuous norms. The simplest case arises for the sup norm on $\cV$
and pointweise collocation, i.e. strong discretization of the data by pointwise
evaluation
on finite sets getting dense in the domain. But it holds also for weak
discretizations and the $L_2$ norm.
\biglf
Unfortunately, a problem for the applications in Sections
  \RSref{SecExa} and \RSref{SecOBC} to quasilinear systems
  arises with \eref{eqvVDens}, because the PDE theory 
  there needs $\cV$ to carry a Hölder norm. But in
  \RScite{schaback:2015-4} and in the
  numerical paper \RScite{boehmer-schaback:2017-1}, the range space
  should be  $\cV=C(\overline\Omega)\times C(\partial\Omega)$,
  making \eref{eqvVDens} obvious.
  This is why we have to change
  the argument in \RScite{schaback:2015-4} accordingly. 

  We assume a linear operator equation of the form
  $Au=f$ with $A\;:\;\cU\to \cV$ and a well-posedness inequality
  $$
  \|u\|_\cU\leq C_S\|Au\|_\cV \fa u\in \cU.
  $$
  The space $\cV$ should carry a norm $\|.\|_{\infty,\cV}$ that maybe weaker
  than
  the norm in $\cV$, and that allows MRD discretizations as in
  \RScite{schaback:2015-4}, i.e.
    $$
\|v\|_{\infty,\cV}=\displaystyle{\sup_{s>0} \|T_sv\|_{\infty,V_s}   }\fa v\in \cV
$$
instead of \eref{eqvVDens}.
This is true if $\cV$ carries a Hölder norm
of type $C^\gamma$ with $\gamma >0$ to satisfy the requirements of PDE theory,
and if collocation is used on the numerical side,
leading to the weaker norm $\|.\|_{\infty,\cV}=\|.\|_\infty$.

  For any given finite-dimensional trial subspace $U_r\subset\cU$,
  we form the finite-dimensional subspace $W_r:=A(U_r)$ and follow the
  argument in \RScite{schaback:2015-4} to find a MRD discretization
  with sup norms on the spaces $V_s$ with
  $$
  \|w_r\|_{\infty,\cV}\leq 2\|T_sw_r\|_{\infty,V_s} \fa w_r\in W_r=A(U_r).
  $$
  On the finite-dimensional subspace $W_r$, we have a norm-equivalence relation
  $$
  c^{-1}_{r,V}\|w_r\|_{\cV}\leq  \|w_r\|_{\infty, \cV}
  \leq  C_{V}\|w_r\|_{\cV} \fa w_r\in W_r,
  $$
  where the left-hand constant may depend on $r$.
  Then, using \eref{eqvVDens}, we get 
  $$
\begin{array}{rcl}
  \|u_r\|_{\cU}
  &\leq &
  C_S\|A(u_r)\|_{\cV}\\
  &\leq &
  c_{r,V}C_S\|A(u_r)\|_{\infty, \cV}\\
  &\leq &
  2 c_{r,V}C_S\|T_s(A(u_r))\|_{\infty, V_s}\\
  &\leq &
  2 c_{r,V}C_S\|A(u_r)\|_{\infty, V}\\
  &\leq &
  2 c_{r,V}C_SC_V\|A(u_r)\|_{\cV}\fa u_r\in U_r,
\end{array}
$$
i.e. the linear problem has a stability bound that depends only on
the trial space, and in a controllable way. On the downside,
the final convergence rates will be decreased by the behavior of $c_{r,V}$.

\begin{Theorem}\RSlabel{TheUnifStabLin}
  Assume a linear operator equation of the form
  $Au=f$ with $A\;:\;\cU\to \cV$ and a well-posedness inequality
  $$
  \|u\|_\cU\leq C_S\|Au\|_\cV \fa u\in \cU.
  $$
  Then for each trial space $U_r\subset\cU$
  there is an MRD discretization by
  uniformly bounded test maps $T_s$ 
  such that the linear problem has a stability bound 
  in the sense of \eref{eq1b} with 
  $$
  C_{SF}(r,s(r))\leq 2 c_{r,V}C_SC_V.\qed
  $$ 
\end{Theorem}
Note that there is
no ellipticity assumption. The basic proof ingredient in
\RScite{schaback:2015-4} is a covering argument
in $\cV$ for the unit ball of the finite-dimensional subspace $A(U_r)$.
On the downside, Approximation Theory shows that a strong amount of oversampling
may be needed for badly chosen collocation points, up to $\dim T_{s(r)}\geq c\,(\dim U_{r})^2$. 

In the special situation of this paper, we have the model situation
$$
  c^{-1}_{r,\gamma}\|w_r\|_{C^\gamma}\leq  \|w_r\|_{\infty}
  \leq  C_{\gamma}\|w_r\|_{C^\gamma} \fa w_r\in W_r
$$
  for the transition between Hölder space $C^\gamma(\Omega)$ and
  $C(\Omega)$ with the sup norm. There does not seem to be any literature
  on this for $0<\gamma<1$, not even for
  simple trial spaces, while for positive integer $\gamma$,
  such inequalities are of  Markov type, see \RScite{markov:1916-1}
  for the univariate polynomial case. But by monotonicity arguments,
  one can replace the constant $c_{r,\gamma}$ for $0<\gamma<1$
  by $c_{r,1}$ at a certain loss that needs further research.
  \biglf
  However, we
  add a case for kernel-based trial spaces. 
  Assume a scale $\{W_r\}_{r>0}$ of trial spaces $W_r$
  consisting of translates of the
  Whittle-Mat\'ern kernel generating Sobolev space $W_2^m(\R^d)$ with
  $m>d/2+\gamma\geq 1+\gamma$,
  and let the translates be formed by sets 
  $X_r:=\{x_1,\ldots,x_{M(r)}\}\subset\overline\Omega$ of $M(r)$ centers that are
  {\em asymptotically uniformly distributed}, i.e. the
  {\em fill distance}
  $$
h_r:=\displaystyle{\sup_{y\in\Omega}\min_{x_j\in X_r} \|y-x_j\|_2}
  $$
  and the {\em separation distance} 
  $$
  q_r:=\displaystyle{ \frac{1}{2}\min_{x_j\neq x_k\in X_r}\|x_j-x_k\|_2}
$$
satisfy
$$
0<c\,q_r\leq h_r\leq C \;q_r
$$
with constants that are independent of $r$. Up to a factor, this
implies $M(r)\approx h_r^{-d}$.
\begin{Theorem}\RSlabel{TheHoelderInverse}
  Under the above assumptions, and if the domain has a $C^1$ boundary,
  $$
  c_{r,\gamma}\leq Ch_r^{-\gamma-d/2}
  $$
  with a constant independent of $r$.
\end{Theorem} 
\begin{proof}
The paper
\RScite{schaback-wendland:2002-1} proves that all trial functions $w_r\in W_r$
satisfy an {\em inverse inequality} 
$$
\|w_r\|_{W_2^m(\Omega)}\leq Ch_r^{-m+d/2}\|w_r\|_{2,X_r}\leq Ch_r^{-m}\|w_r\|_{\infty,X_r}
$$
with generic constants depending on $m$ and the domain,
but not on $r$ and the position
of centers in $X_r$.
\biglf
By Morrey's embedding theorem, 
Hölder spaces $C^{0,\gamma}(\Omega)$ for $0<\gamma<1$
are continuously embedded into Sobolev space $W_p^1(\Omega)$
if $p=d/(1-\gamma)\in
(d,\infty)$ and if the domain has a $C^1$ boundary. 
\biglf
Then we invoke a sampling inequality \RScite{arcangeli-et-al:2012}
$$
|u|_{W_p^1(\Omega)}\leq
C\left(h_r^{m-1-d(1/2-1/p)_+}|u|_{W_2^m(\Omega)}+h_r^{-1}\|w\|_{\infty,X_r}\right)
\fa u\in W_2^m(\Omega)%RS ,\;0\leq 1\leq \lceil m-d/2\rceil-1
$$
and get
$$
|w_r|_{W_p^1(\Omega)}\leq
Ch_r^{-1-d(1/2-1/p)_+}\|w_r\|_{\infty,X_r}\fa w_r\in W_r.
$$
For $d\geq 2$ we have $1/p\leq 1/2$ and thus finally arrive at
$$
|w_r|_{C^{0,\gamma}(\Omega)}\leq Ch_r^{-\gamma-d/2}\|w_r\|_{\infty,X_r}\fa w_r\in W_r.
$$
\end{proof}
An extension to nonlinear problems is 
\begin{Theorem}\RSlabel{TheUnifStab}
  Assume a well-posed nonlinear problem satisfying
  Theorem \RSref{thePDELinWP} and consider it in a neighborhood
  of a solution $u^*$.
  % consider the fixed linearization at the true solution $u^*$.
  Then for each trial space $U_r\subset\cU$,
  there is an MRD discretization by
  uniformly bounded test maps $T_s$ in the sense of \RScite{schaback:2015-4}
  such that the nonlinear problem is uniformly stable
  in the sense of \eref{eq1b} with an upper bound as in
  Theorem \RSref{TheUnifStabLin}. 
\end{Theorem}
\begin{proof}
In view of Theorem \RSref{thePDELinWP} we can assume
that all local inverses of the Fr\'echet derivatives
have a uniform upper bound, and we can reuse the constant $c_F$ of
\eref{eq1a} to get 
$$
\|v\|_\cU\leq c_F\|F'(u^*)v\|_\cV \fa v\in \cU
$$
for the linearization $F'(u^*)$ of $F$  in $u^*$.
If this linear problem is given an MRD discretization in the sense
of \RScite{schaback:2015-4}, we have for each trial space $U_r$
a scale of uniformly bounded 
test maps $T_s$ such that 
$$
\|u_r\|_\cU\leq 2 c_{r,V}c_FC_V \|T_sF'(u^*)u_r\|_{V_s} \fa u_r\in U_r
$$
holds. Note, however, that \RScite{schaback:2015-4}
lets $T_s$ depend on the linear PDE problem, while $U_r$ is fixed.
Thus we have to fix $T_s$ by using the fixed linearization in $u^*$.
Note further that Definition 3 and (10) in \RScite{schaback:2015-4}
imply that the maps $T_s$ are uniformly bounded, if an MRD discretization
is chosen. 
\biglf
We now extend the above bound to linearizations at other functions.
With $K(r):=c_{r,V}c_FC_V$ we get 
$$
\begin{array}{rcl}
  \|u_r\|_\cU
  &\leq &
  2K(r)\|T_s{F'(u^*)}u_r\|_{V_s}\\
  &\leq &
  2K(r)\|T_s({F'(u^*)}-{F'(u)})u_r\|_{V_s}+2K(r)\|T_s{F'(u)}u_r\|_{V_s}\\
  &\leq &
  2K(r)C\|T_s\|\|u-u^*\|_\cU\|u_r\|_\cU+2K(r)\|T_s{F'(u)}u_r\|_{V_s}
\end{array}
$$
to arrive at 
$$
\|u_r\|_\cU\leq 4K(r)\|T_s{F'(u)}u_r\|_{V_s} \fa u_r\in U_r
$$
and all $u$ in a neighborhood of $u^*$, using
uniform boundedness of the test maps $T_s$ again.
We finally repeat the proof
of Theorem 4 of \RScite{BoScha12},
noting that due to uniform boundedness of the $T_s$
we have uniformly bounded
$C''(s)$ by equation (25) and well behaving $R$ by the first formula below
  (27) there. Then,
up to a constant, the stability property of the linearization
carries over to the nonlinear case. 
\end{proof}
Standard cases of MRD discretizations are collocation methods
where the operator values are sampled in sufficiently many points.
But \RScite{schaback:2015-4} also treats more sophisticated situations 
that we do not pursue here.
%%%%%%%%%%%%%%%%%%%%%%%%%
%****************************************************************
\section{Error Bounds and Convergence Rates}\RSlabel{SecEBCR}
Since the previous section guaranteed stability inequalities
by a suitable test discretization for any choice of trial spaces,
\RScite{BoScha12} implies that the error for the numerical solution
along the lines of Section
\RSref{SecTestSol} inherits the behavior \eref{eqepsr}
of the approximation error
up to the factor $c_{r,V}$.
This applies to many different situations,
depending on the trial space and the smoothness assumptions
on the solution, the domain, and the PDE problem as a whole.
In cases where the true solution has a rapidly convergent
expansion into trial functions, this convergence rate,
measured in the norm of the well-posedness property
of the problem, carries over to the numerical solution,
provided that an MRD test strategy is chosen,
including sufficient oversampling, and
if a nonlinear optimization like in Section \RSref{SecTestSol}
is carried out. Typical numerical examples
were provided in \RScite{BoScha12} and \RScite{boehmer-schaback:2017-1},
the latter focusing on the Monge-Amp\'ere equation that
is not covered by this paper unless a result like
Theorem \RSref{thePDELinWP} is provided.
\biglf
  To make this paper self-contained, we state
  the final result of 
  \RScite{BoScha12} adapted to the situation here, including Theorem \RSref{TheUnifStab}:
  \begin{Theorem}\RSlabel{Theerrbnd}
For a well-posed nonlinear problem \eref{eqFGB}
in the sense of Section \RSref{ss:WPP} with a unique true solution $u^*\in \cU$,
and for all finite-dimensional trial spaces $U_r\subset\cU$ there is a testing
strategy using a test map $T_s$ such that the solution technique of
Section \RSref{SecTestSol} leads to a numerical solution $\tilde u_r\in U_r$
with an error bound
$$
\|u^*-\tilde u_r\|_\cU\leq C \epsilon(r,u^*)c_{r,V}
$$
where $\epsilon(r,u^*)$ is the error of a good approximation $u_r^*\in U_r$ to
$u^*$ as given in \eref{eqapperr}, where 
  $C$ depends only on the well-posedness of $F$ near $u^*$m, and where
  $c_{r,V}$ is the instability factor induced by being forced to a Hölder-type
  treatment of the PDE. In short, the
convergence rate of the approximation error
in $U_r$ determines the convergence rate of the
solution to the PDE problem up to the factor $c_{r,V}$. $\qed$
\end{Theorem} 
Note that the above result allows various convergence rates, depending on the
smoothness of $u^*$ and the trial space $U_r$, up to spectral convergence.
%%%%%%%%%%%%%%%%%%%%%%%%%%%%%%%%%%%%%%%
\section{Quasilinear Equations of Order Two}\RSlabel{SecExa}
We now show that the above assumptions
are satisfied for 
quasilinear elliptic  second order equations
in strong  form  on a bounded Lipschitz 
domain, and we start with the case of Dirichlet boundary conditions.
\biglf
We use results from Gilbarg
and Trudinger's book \RScite{GiTr01}, Chapters 10 and 15, as summarized in
\RScite{Boehmer05}, subsection 2.5.4,
Theorems 2.61 and 2.64 on existence, uniqueness, and regularity. We skip over
the nonuniform elliptic cases, cf. (2.217), (2.218) there, noting that
subsections
2.6.4, 2.6.6,.2.6.7 of \RScite{Boehmer05}
would allow strong extensions to systems of
$q$ equations of order $2m$ with higher technical complications.
\biglf
The spaces and operators are
\bql{eqQLP}
\begin{array}{rcl}
\cU&:=& C^{2,\gamma}(\overline{\Omega},\R),\\
\cV&:=&V_1 \times V_2 = C^\gamma (\overline{\Omega}) × C^\gamma (\partial\Omega),\\
Gu &:=&\displaystyle{ \sum_{i,j=0}^d
  a_{ij}(x, u, \nabla u)\partial^i \partial^j u},\\
Bu &:=&u_{|_{\partial \Omega}}
\end{array} 
\eq
where we used $-\partial^0 := Id$  and impose the compatibility condition
$$
\cD(G) := \left\{u \in \cU \;:\;(x, u, \nabla u) \in \cD(a_{ij} ),
0 \leq i, j \leq d, Gu \in C^\gamma (\overline{\Omega})\right\}
$$
to make everything well-defined. Furthermore, the domain boundary should
satisfy $\partial\Omega \in C^{2,\gamma}$.
\biglf
The linearization around a function $u$ can be formally written as an operator
\bql{eqGQlinRS}
\begin{array}{rcl}
G' (x, z, p)v
&=&\displaystyle{\sum_{i,j=0}^d  {a_{ij}(x, z, p)}\partial^i \partial^j v}\\
&&+ \displaystyle{ \sum_{i,j=0}^d \partial^i \partial^j u
  \left(v\frac{\partial}{\partial z}+\sum_{k=1}^d\partial^k v
  \frac{\partial}{\partial p_k}   \right)a_{ij}(x, z, p)} 
\end{array} 
\eq
for all points $(x, z, p)$ in a neighborhood of {$(x, u^*(x), \nabla u^*(x))$,}
and we assume the
principal part to be uniformly elliptic there.
{Furthermore, the coefficients 
  $a_{ij},\frac{\partial}{\partial z}a_{ij}, \frac{\partial}{\partial p_k} a_{ij}$
  should be Lipschitz continuous in $z,p$ near
  the locally unique solution $u^*$, i.e.  $(x,u^*(x), \nabla u^*(x))$,
and finally satisfy the conditions around \cite[Thm. 2.61]{Boehmer05}.}
\begin{Theorem}\RSlabel{thePDELinWP}
  Under {these} %RS the above
  assumptions,
  %RS and under growth and regularity conditions that we omit here,
  the following holds:
\begin{enumerate}
\item There exists a solution $u^*\in C^{2,\gamma}(\overline{\Omega})$
  of the problem $Fu^* = f$.
% in $\Omega$, and $u^*_{|_{\partial\Omega}} = f_2$  on $ \partial\Omega$.
\item
  The principal part of the linearization $G'$ is coercive near
  $u^*$.
\item
  If we write $F = (G, B)$ as in \eref{eqFGB}, {and if zero is
  not an eigenvalue of
  $F'(u^*)$,} then $F'$ is boundedly invertible near $u^*$
  and $u^*$is a locally unique
  solution of \eref{eqFGB}.
\item In a neighborhood of $u^*$, the linearization in $u$
  is Lipschitz continuous in $u$.
\end{enumerate}
\end{Theorem} 
More detailed existence, uniqueness  and regularity results are
listed in  \RScite{Boehmer05} near Theorems 2.61 {and } 2.64.
\biglf
%****************************************************************
\section{Non-Dirichlet Boundary Conditions}\RSlabel{SecOBC}
In this section we show that for quasilinear elliptic differential
equations  Dirichlet and generalized  Neumann, Robin, or
mixed  boundary conditions
define well--posed problems. Then, by Theorem \RSref{Theerrbnd},
suitably oversampled meshless methods are convergent.
This generalization  is highly nontrivial, because uniqueness might be missing
in the basic theory. 
Furthermore, the  boundary operators have   to match the
differential operators by {\em complementing conditions}.
These have been intensively studied, e.g. by
Agmon/ Douglis/Nirenberg \cite{AgDoNi59},
Lions, Magenes \cite{LiMa72},   Wloka \cite{Wloka82},
Zeidler \cite{Zeidler86} and  Amann \cite{Amann90} p. 21. 
\biglf
To describe what is possible, we go back to the notation
in Section \RSref{SecExa}.
We assume that $A$, mimicking  
the principal part of
$G'(u^*)$ in the first line of \eref{eqGQlinRS},
is a strongly elliptic linear differential operator.  
\biglf
Neumann/Robin
boundary operators take first derivatives of $u\in C^{2,1}(\overline{\Omega})$
on the $C^{1,1}$ boundary, and we can parametrize them into a normal component
$\frac{\partial}{\partial \nu}$ and $d-1$ tangential components
$\frac{\partial}{\partial t^i},\;1\leq i\leq d-1$ to arrive at
\bql{NRbc}
B_{NR}u:=b_\nu \frac{\partial u}{\partial \nu}
+\displaystyle{\sum_{i=1}^{d-1}b_i\frac{\partial u}{\partial t^i}  {+ b_0 u}} 
\eq
with continuous functions $b_\nu$ and $b_i,\;
0 % \blue{  0  \emph{ instead of  } 1}
\leq i \leq d-1$.
A crucial assumption then is that $b_\nu$ is strictly positive.
\biglf
Mixed conditions can be written as
$$
B_\delta u:=\delta B_{NR}u+(1-\delta)B_Du 
$$
with $\delta \;:\;C(\partial\Omega)\to \{0,1\}$ being a
piecewise constant
switch function on the boundary with
constant value $\delta(\Gamma)$
on different components $\Gamma \subset \pa\Omega$.
If $A$ is now any strongly elliptic
linear differential operator, the operator pairs
$(A,B_D),\;(A,B_{NR})$, and $(A,B_\delta)$ are coercive.
\biglf
This can be transferred to the nonlinear situation.
We summarize \cite{Boehmer18} into
\begin{Theorem}\RSlabel{TheBCstuffRS}
  If the principal part $A$ of the linearization
  $G'(u^*)$ %  of $A$ %the quasilinear PDE \eref{eqQLP} 
  is uniformly elliptic,
  and if the above conditions on the boundary operators are satisfied,
  then the nonlinear problems $(G,B_D), (G,B_{NR})$, and $(G,B_\delta)$ are 
  well-posed in the sense of Section
  \RSref{ss:WPP},
  provided that $\lambda =0$ is not an eigenvalue of the linearized problems.
\end{Theorem}
\begin{Theorem}\RSlabel{TheBCBifRS}
  Numerical Liapunov-Schmidt and center manifold methods for the above wide range of nonlinear elliptic (and parabolic)  problems $(G,B)$ in a bifurcation  point $u_0$ are defined by bordering the nonlinear elliptic problem by a few rows and columns incorporating 
$\cN(G'(u_0)$ and $(\cR(G'(u_0))^\perp.$ Thus these extended systems are well posed, cf.  \cite{Boehmer05,Boehmer06,Boehmer93,Boehmer99a,Boehmer99b,Boehmer17}. 
Assume the conditions in Theorem \RSref{TheBCstuffRS}.  Then again the previous
results for quasilinear equations  apply and prove stable and convergent
meshfree methods %RS MFMs
for  %RS \red{their}
Liapunov-Schmidt and center manifold techniques. %RS methods.
\end{Theorem} 
This builds on results
of \RScite{Amann90} for parabolic problems,
but we have to refer the reader to
\RScite{Boehmer18} for details.
An extension to nonlinear boundary  operators is obvious.
  Just replace in \eref{NRbc} the $b_\nu$ and $b_i$ by $b_\nu(u)$ and $b_i(u)$
  and thee previous $F=(G,B_{NR})$ by $F(u)=(G(u),B_{NR}(u)u)$.
  Then all results prevail. 
%****************************************************************
\section{Summary and Outlook}\RSlabel{SecSumOut}
Our nonlinear discretization theory splits the necessary work
for error bounds and convergence rates into two parts:
\begin{enumerate}
 \item into PDE theory for establishing explicit results
   about well-posedness,
 \item into Approximation theory for proving rates
   of convergence in trial spaces.
\end{enumerate} 
Once these ingredients are provided, there are
numerical methods that have  a certain stability, though at the expense of
oversampling and nonlinear optimization. They guarantee that the numerical
solution of the PDE solution converges at the rate of the trial space
approximation reduced by a factor
  that arises from the Hölder theory
  of the PDE. Furthermore, they apply to all possible trial spaces
and usually do not require any background integration.
\biglf
However, there are
several shortcomings that need additional work:
\begin{enumerate}
\item The amount of oversampling is strongly problem-dependent
  and cannot easily be addressed \RScite{schaback:2015-4}.
\item Playing the algorithmic part back to optimization
  \RScite{BoScha12, boehmer-schaback:2017-1} ignores numerical efficiency.
\item There is no mention of sparsity considerations.
\item The convergence rates will in many cases not be competitive
  with sophisticated finite element techniques. They require
  high smoothness of the true solution to be effective. But,
  on the positive side, the optimization techniques are
  very easy to implement compared to finite element methods
  \RScite{BoScha12, boehmer-schaback:2017-1}, making them attractive
  for users who want quick answers without too much hassle.
\item There may be a different workaround for the stability complications
  arising from the Hölder theory of the PDE. If the well-posedness
norm in $\cU$ is taken weaker than the full norm on $\cU$, the approach gets
closer to \RScite{schaback:2015-4}
and yields uniform stability in that weaker norm,
but then the linearization arguments of
\RScite{BoScha12} have to be replaced.
\end{enumerate} 
%%%%%%%%%%%%%%%%%%%%%%%%%%%%%%%%%%%%%%%%%%%%%%%%
%%%%%%%%%%%%%%%%%%%%%%%%%%%%%%%%%%%%%%%%%%%%%%%
%%%%%%%%%%%%%%%%%%%%%%%%%%%%%%%%%%%%%%%%%%%%%%%%
%%%%%%%%%%%%%%%%%%%%%%%%%%%%%%%%%%%%%%%%%%%%%%%
%%%%%%%%%%%%%%%%%%%%%%%%%%%%%%%%%%%%%%%%%%%%%%%
\bibliographystyle{plain}
%\bibliography{jointbib,bigbib}

\end{document}